\documentclass[11pt]{amsart}
\usepackage[top=1.2in,bottom=1.2in,left=1.2in,right=1.2in]{geometry}
\usepackage[inline]{enumitem}
\usepackage{amsthm}

\usepackage{graphicx}
\usepackage{tikz}
\usetikzlibrary{positioning}
\usetikzlibrary{arrows}
\usetikzlibrary{matrix,arrows,decorations.pathmorphing}

\usepackage{mathtools}

\usepackage{float}
\usepackage{amssymb}
\usepackage{xypic,pinlabel}
\usepackage{hyperref}
\hypersetup{
    colorlinks,
    citecolor=black,
    filecolor=black,
    linkcolor=black,
    urlcolor=black
}

\newcommand{\D}{\mathbb{D}}

\newcommand{\Z}{\mathbb{Z}}

\newcommand{\B}{B}

\DeclareMathOperator{\Mod}{Mod}

\newcommand{\p}[1]{\bigskip \noindent \emph{#1}.}

\theoremstyle{plain}

\newtheorem{theorem}{Theorem}
\newtheorem{proposition}{Proposition}

\theoremstyle{definition}

\newcommand{\nc}{\newcommand}
\nc{\dmo}{\DeclareMathOperator}
\nc{\para}[1]{\medskip\noindent\textbf{#1.}}

\title{The commutator subgroup of the braid group is generated by two elements}

\author{Kevin Kordek}

\address{Kevin Kordek \\ School of Mathematics\\ Georgia Institute of Technology \\ 686 Cherry St. \\ Atlanta, GA 30332}
\email{kevin.kordek@math.gatech.edu}

\thanks{The author was supported by National Science Foundation Grant No. DMS - 1057874.}

\begin{document}
\maketitle

\begin{abstract}
For $n$ at least 7 and $n$ equal to 5, we give generating sets of size 2 for the commutator subgroup of the braid group on $n$ strands. These generating sets are of the smallest possible cardinality. For $n$ equal to 4 or 6, we give generating sets of size three. We also prove that the commutator subgroup of the braid group on 4 strands cannot be generated by fewer than three elements. 
\end{abstract}

\section{Introduction}

Let $\sigma_1, \ldots, \sigma_{n-1}$ denote the standard Artin generating set for the braid group $\B_n$ on $n$ strands. Although the size of this generating set increases with $n$, for $n\geq 3$ there is another generating set for $\B_n$ consisting only of $\sigma_1$ and the periodic element
\[
\alpha = \sigma_1\cdots \sigma_{n-1}.
\]
This generating set is minimal when $n\geq 3$, since $\B_n$ is not cyclic for these values of $n$.

Let $\B_n'$ denote the commutator subgroup of $\B_n$. This group is trivial when $n =2$ and is non-trivial for all $n\geq 3$. The group $\B_3'$ is isomorphic to a free group of rank $2$ (see~\cite[p.6]{lin}), and so can also be generated by two elements. For $n\geq 4$, $\B_n'$ is no longer free and Gorin--Lin~\cite{gorinlin, lin} found a generating set for $\B_n'$ of size $n$. Chen, Margalit, and the author later pointed out~\cite{chenkordekmargalit} that it immediately follows from their presentation that $\B_n'$ is generated by $n-2$ elements when $n \geq 5$, namely the set of all elements $\sigma_1\sigma_i^{-1}$ with $2\leq i \leq n-1$. 

One might first expect that, by analogy with the situation for $\B_n$, there is a small generating set for $\B_n'$ consisting of a fixed, finite collection of elements of the form $\sigma_1\sigma_i^{-1}$ and periodic elements of $\B_n$. However, this is impossible because $\B_n'$ does not contain any periodic elements. Nevertheless, we construct generating sets for $\B_n'$ of size 2 that consist an element of the form $\sigma_1\sigma_i^{-1}$ and an element that behaves very much like a periodic braid. 

\begin{theorem}\label{thm:main}
Let $n=5$ or $n\geq 7$.  For any integer $2\leq k\leq n-2$ such that $\gcd(k,n) = 1$, the group $\B_n'$ is generated by the two elements
\[
\alpha^k\sigma_1^{-k(n-1)}\quad \text{and} \quad \sigma_1\sigma_{1+k}^{-1}.
\]

\end{theorem}

The commutator subgroup $\B_n'$ is not a cyclic group for $n\geq 3$ (in particular, it is not cyclic for $n\geq 5$), and so the generating sets given by Theorem~\ref{thm:main} are of the smallest possible cardinality.

\medskip

 Our proof of Theorem~\ref{thm:main} breaks down when $n\in \{3,4,6\}$, as in these cases any integer coprime to $n$ is congruent to either $1$ or $-1\mod n$. That these are the only values of $n\geq 3$ for which this situation occurs follows from elementary number theory. As mentioned above, $\B_3'$ can be generated by two elements. We will show that both $\B_4'$ and $\B_6'$ admit generating sets of size 3, and that $\B_4'$ cannot be generated by fewer than three elements.  We do not know if $\B_6'$ can be generated by fewer than three elements.

 \begin{theorem}\label{thm2}
 \
 \begin{enumerate}
 \item  The group $\B_4'$ cannot be generated by fewer than three elements. There is a generating set for $\B_4'$ consisting of the three elements $\sigma_1\sigma_2^{-1},\sigma_1(\sigma_1\sigma_2^{-1})\sigma_1^{-1}$, and $\sigma_1\sigma_3^{-1}$.
  
  \vspace{.1in}
  
 \item The group $\B_6'$ is generated by the three elements $\sigma_1\sigma_2^{-1},\sigma_1\sigma_3^{-1}$, and $\alpha^2\sigma_1^{-10}$. 
 \end{enumerate}
 \end{theorem}
 
 We point out that the purported generators from Theorems~\ref{thm:main} and~\ref{thm2} do indeed lie in $\B_n'$, as their signed word lengths with respect to the $\sigma_i$ are equal to $0$. 

\p{Prior results} For $n\geq 4$, Gorin--Lin~\cite[p.7]{lin} found a finite presentation for $\B_n'$ with $n$ generators. Chen, Margalit, and the author~\cite[Lemma 8.6]{chenkordekmargalit} observed that only a subset of their generators were needed, and deduced that $\B_n'$ is generated by the $n-2$ elements $\sigma_1\sigma_i^{-1}$ with $2\leq i\leq n-1$. 

There are many results concerning minimal generating sets of mapping class groups of higher genus surfaces; we refer the reader to the paper by Baykur--Korkmaz~\cite{baykurkorkmaz} and the references therein. For $g\geq 3$, the genus $g$ mapping class group $\Mod(S_g)$ is a perfect group and hence equal to its own commutator subgroup. Thus for $g\geq 3$, every mapping class can be expressed as a product of commutators. Baykur--Korkmaz~\cite{baykurkorkmaz} have shown that $\Mod(S_g)$ is generated by two commutators for $g\geq 5$ and that it is generated by three commutators for $g\geq 3$. Since $\Mod(S_g)$ is not cyclic, their generating set is of the smallest possible cardinality when $g\geq 5$. 

The generating sets from Theorem~\ref{thm:main} are similar to the generating set found by Baykur--Korkmaz, (which consists of a product of commuting Dehn twists and a periodic element; see~\cite[p.2 ]{baykurkorkmaz}), in that they are of size 2 and have the property that one of the generators is a product of commuting half-twists and the other behaves much like a periodic element. Also, just as $\Mod(S_g)$ is perfect for $g\geq 3$, the group $\B_n'$ is perfect for $n\geq 5$. However, our result differs from theirs in that only the generator $\sigma_1\sigma_{1+k}^{-1}$ is known to be a commutator in $\B_n'$; we do not know if $\alpha^k\sigma_1^{-k(n-1)}$ is a commutator. 

\subsection*{Outline} After recalling some basic facts about braid groups and mapping class groups, we give the proofs of Theorem~\ref{thm:main} and Theorem~\ref{thm2} in Section~\ref{sec:proofs}. In the appendix, we give a self-contained proof that $\B_n'$ is generated by all elements of the form $\sigma_1\sigma_i^{-1}$ for $n\geq 5$. 

\subsection*{Acknowledgements} The author would like to thank Dan Margalit for suggesting the problem, for helpful conversations, and for numerous comments on early drafts of this paper. The author is grateful to the anonymous referee for their careful reading of the paper and for providing excellent comments that greatly enhanced its quality. 

\section{Proofs of the theorems}\label{sec:proofs}

Before we can prove Theorems~\ref{thm:main} and~\ref{thm2}, we require some setup. We will identify $\B_n$ with the mapping class group of the closed disk with $n$ marked points $\D_n$, which is the group of homotopy classes of homeomorphisms of $\D_n$ that fix the boundary pointwise and preserve the set of marked points. We take the marked points to be arranged in a circle centered at the origin in $\D_n$. They will be labeled with elements of $\Z/n$. 

For each $i\in \Z/n$ we define $c_i$ to be the simple closed curve pictured in Figure~\ref{fig} which surrounds the $i$th marked point and the $(i+1)$st marked point. Let $H_{c_i}$ denote the positive half-twist supported in the closed disk bounded by $c_i$. Then for $i\neq 0$ we have under our identification of $\B_n$ with $\Mod(\D_n)$ that 
\[
\sigma_i = H_{c_i}.
\]
For $i = 0$, we define $\sigma_0 = H_{c_0}$. 

Under the identification of $\B_n$ with $\Mod(\D_n)$, the element $\alpha$ corresponds to a counter-clockwise rotation of the interior of the disk by $2\pi/n$; see Figure~\ref{fig}. The rotation $\alpha$ satisfies $\alpha(c_i) = c_{i+1}$, and so we have for all $i$ and $m$ that

\[
\alpha^m\sigma_i\alpha^{-m} = \sigma_{i+m}. 
\]

The curves $c_i$ and $c_j$ are disjoint if and only if $i-j \neq \pm 1$ mod $n$, as the latter condition is equivalent to the statement that $c_i$ and $c_j$ surround distinct pairs of marked points. It follows that 
\[
\sigma_i\ \text{commutes with}\ \sigma_j\quad \Longleftrightarrow \quad i-j \neq \pm 1\mod n.
\]

\medskip

We are now ready to prove Theorems~\ref{thm:main} and~\ref{thm2}.

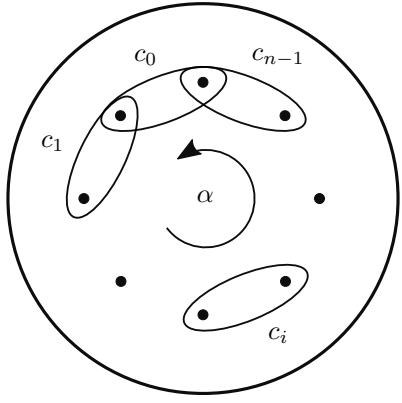
\begin{figure}
\labellist
\small\hair 2pt
\pinlabel $c_{n-1}$ at 410 510
\pinlabel $c_0$ at 210 510
\pinlabel $c_1$ at 70 380
\pinlabel $c_i$ at 410 90
\pinlabel $\alpha$ at 300 300
\endlabellist
\includegraphics[scale=.25]{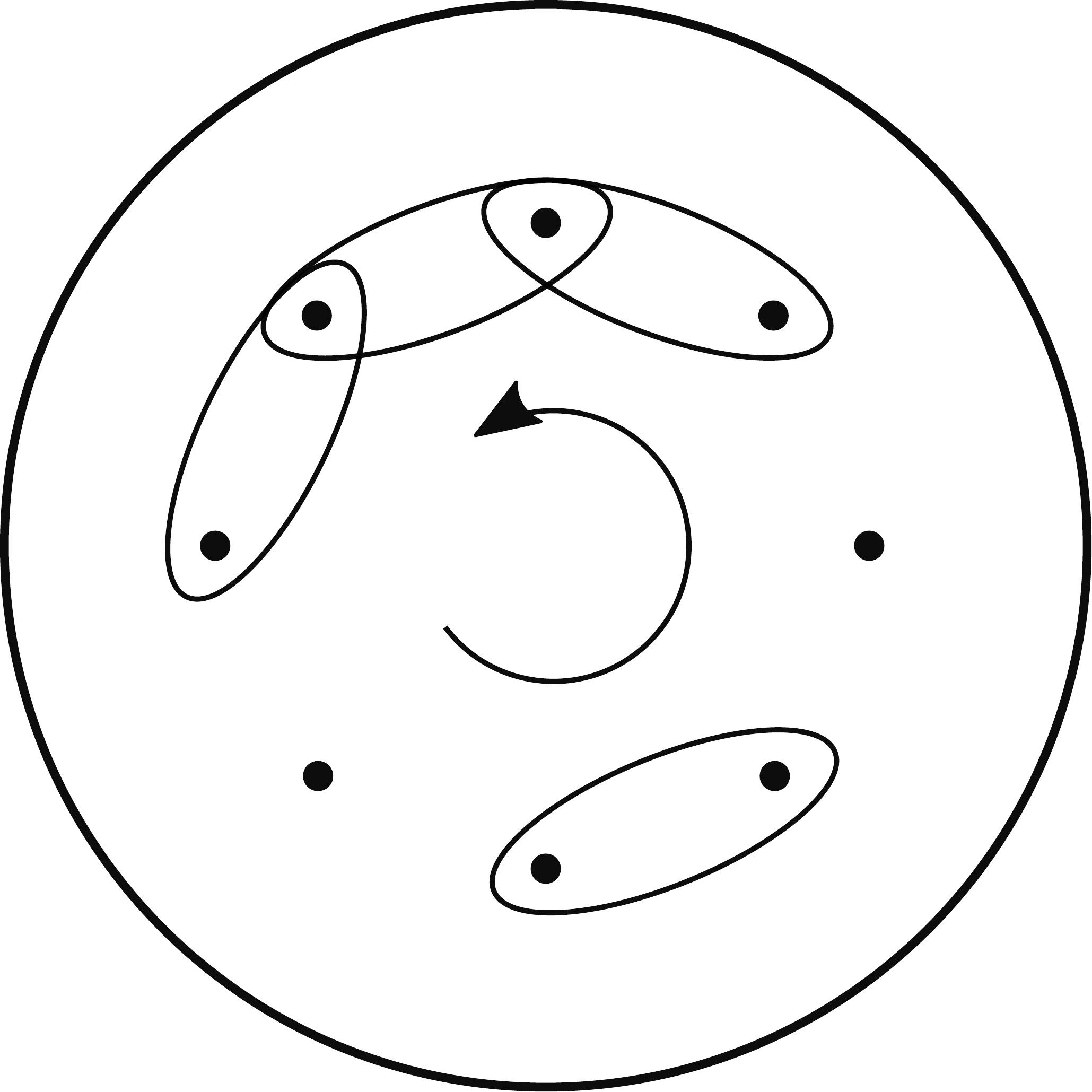}
\caption{The $n$ marked points in the disk $\D_n$; the rotation $\alpha$; the curves $c_i$.}
\label{fig}
\end{figure}

\begin{proof}[Proof of Theorem~\ref{thm:main}]
As we mentioned in the introduction, $\B_n'$ is generated by the elements $\sigma_1\sigma_i^{-1}$ with $2\leq i \leq n-1$. We will prove Theorem~\ref{thm:main} by showing that each of the generators $\sigma_1\sigma_i^{-1}$ is an element of the subgroup generated by $\sigma_1\sigma_{1+k}^{-1}$ and $\alpha^k\sigma_1^{-k(n-1)}$.

To this end, for all $i\geq 0$ we define auxiliary elements of $\B_n'$
\[
r_i = \alpha^k\sigma_{1+ik}^{-N} \qquad \qquad s_i = \sigma_{1+ik}\sigma_{1+(i+1)k}^{-1}
\]
where $N = k(n-1)$. Observe that $r_0=\alpha^k\sigma_1^{-N}$ and that $s_0=\sigma_1\sigma_{1+k}^{-1}$. We have that $r_i = \alpha^{ki}(r_0)\alpha^{-ki}$ and $s_i = \alpha^{ki}(s_0)\alpha^{-ki}$, and so it follows that $r_i$ and $s_i$ are indeed contained in $\B_n'$ for all $i\geq 0$. 

\medskip

Our proof has two steps:

\begin{enumerate}[leftmargin=*]
\item Prove that both $r_i$ and $s_i$ are contained in $\langle r_0, s_0\rangle$ for all $i\geq 0$.
\item Prove that $\sigma_1\sigma_i^{-1}$ can be written as a product of the $s_j$, and hence that it is an element of $\langle r_0, s_0\rangle$.
\end{enumerate} 

\p{Step 1}  We begin with the first step.  It is trivially true that $r_0$ and $s_0$ are contained in $\langle r_0, s_0\rangle$.
%
Assume now that there exists $m\geq 1$ such that $r_{\ell}$ and $s_{\ell}$ are both contained in $\langle r_0,s_0\rangle$ for all $\ell\leq m$. We will show that 
\[
r_{m+1} =r_ms_m^N \qquad \text{and} \qquad s_{m+1} = r_ms_mr_m^{-1}
\]
and hence that both are contained in $\langle r_0,s_0\rangle$. The claim will then follow by induction. 

\medskip

By the definition of $r_m$ and $s_m$ we have that
\begin{align*}
r_ms_mr_m^{-1} &= (\alpha^k\sigma_{1+mk}^{-N})(\sigma_{1+mk}\sigma_{1+(m+1)k}^{-1})(\sigma_{1+mk}^{N}\alpha^{-k})\\[.1in]
&=\alpha^k\left(\sigma_{1+mk}^{-N}\sigma_{1+mk}\sigma_{1+(m+1)k}^{-1}\sigma_{1+mk}^{N}\right)\alpha^{-k}.
\end{align*}
Since  
\[
(1+(m+1)k)- (1+mk) = k\neq \pm 1,
\]
\medskip
the elements $\sigma_{1+mk}$ and $\sigma_{1+(m+1)k}$ commute and hence
\[
r_ms_mr_m^{-1} = \alpha^k\left(\sigma_{1+mk}\sigma_{1+(m+1)k}^{-1}\right)\alpha^{-k} = \sigma_{1+(m+1)k}\sigma_{1+(m+2)k}^{-1} = s_{m+1},
\]
as desired. 

Since $r_{m+1} = \alpha^k\sigma_{1+(m+1)k}^{-N}$, we may write

\medskip

\begin{align*}
r_{m+1} = \alpha^k\sigma_{1+(m+1)k}^{-N} &= \alpha^k\left(\sigma_{1+mk}^{-N}\sigma_{1+mk}^N\right)\sigma_{1+(m+1)k}^{-N}\\[.1in]
& = \left(\alpha^k\sigma_{1+mk}^{-N}\right)\left(\sigma_{1+mk}^N\sigma_{1+(m+1)k}^{-N}\right).
\end{align*}

\medskip

\noindent Because $\sigma_{1+mk}$ commutes with $\sigma_{1+(m+1)k}^{-1}$, we further have that 

\medskip

\[
\sigma_{1+mk}^N\sigma_{1+(m+1)k}^{-N} = \left(\sigma_{1+mk}\sigma_{1+(m+1)k}^{-1}\right)^N
\]

\noindent and hence that
\begin{align*}
r_{m+1} = \left(\alpha^k\sigma_{1+mk}^{-N}\right)\left(\sigma_{1+mk}\sigma_{1+(m+1)k}^{-1}\right)^N = r_m s_m^N,
\end{align*}

\vspace{.1in}

\noindent also as desired. The claim follows, and Step 1 of the proof is complete. 

\p{Step 2} We now carry out the second step of the proof. Let $2\leq i\leq n-1$. Since $\gcd(k,n) =1$, there exists $\ell\geq 0$ such that $i-1 = k\ell \mod n$. That is, there exists $\ell \geq 0$ such that
\[
i = 1 + k\ell \quad \mod n.
\]

We claim that $\sigma_1\sigma_i^{-1} = s_0 s_1\cdots s_{\ell-1}$. Applying the definition of the $s_j$, we obtain that 
\[
s_0 s_1\cdots s_{\ell-1} = \prod_{j=0}^{\ell-1}\left(\sigma_{1+jk}\sigma_{1+(j+1)k}^{-1}\right)
\]
(here the products are formed from left to right). All of the terms of the product cancel except for the first factor and the last factor. Since the first factor is $\sigma_1$ and the last factor is $\sigma_{1+k\ell}^{-1} = \sigma_i^{-1}$, the product is equal to $\sigma_1\sigma_i^{-1}$, as claimed. Step 2 of the proof is therefore complete, and the theorem follows. 
\end{proof}

We now give the proof of Theorem~\ref{thm2}.
\begin{proof}[Proof of Theorem~\ref{thm2}]

We prove the two statements of the theorem in turn. We begin with the first statement and then prove the second statement.

\medskip 

We claim that $\B_4'$ is generated by the three elements from the first statement of the theorem. We will use the finite presentation for $\B_n'$ found by Gorin--Lin~\cite[p.7]{lin}. Their presentation implies that $\B_4'$ is generated by the four elements
\[
u =\sigma_2\sigma_1^{-1} \qquad v = \sigma_1\sigma_2\sigma_1^{-2} \qquad w =\sigma_2\sigma_3\sigma_1^{-1}\sigma_2^{-1} \qquad c = \sigma_3\sigma_1^{-1}.
\]
Their relation (1.16) states that $w = ucu^{-1}$, and so $\B_4'$ is generated by $u,v$ and $c$. It follows that $\B_4'$ is generated by $u^{-1} = \sigma_1\sigma_2^{-1}, v^{-1} = \sigma_1(\sigma_1\sigma_2^{-1})\sigma_1^{-1}$, and $c^{-1} = \sigma_1\sigma_3^{-1}$, as desired. 

We next claim that $\B_4'$ cannot be generated by two elements\footnote{We remark that Gorin--Lin~\cite[p.7]{lin} (see also~\cite[p.601]{gorinlin}) used the homomorphism $\psi$ to prove that $\B_4'$ is an extension of $F_2$ by $F_2$. This fact can also be used to prove the claim.}. Suppose to the contrary that $\B_4'$ can be generated by two elements. Then there exists a surjective homomorphism 
\[
\phi: F_2\rightarrow \B_4',
\]
where $F_2$ denotes the free group of rank 2. There is also a surjective homomorphism 
\[
\psi: \B_4'\rightarrow \B_3'\cong F_2
\]
obtained by restricting the surjective homomorphism $\B_4\rightarrow \B_3$ defined by sending $\sigma_1$ and $\sigma_3$ to $\sigma_2$ to $\sigma_2$. Since both $\phi$ and $\psi$ are surjective, the composition 
\[
\psi\circ \phi: F_2\rightarrow \B_3' \cong F_2
\] 
is surjective, and hence, by the Hopfian property of free groups, an isomorphism. The fact that $\phi$ is surjective now implies that $\psi$ is injective (and hence an isomorphism). On the other hand, $\sigma_1\sigma_3^{-1}$ is contained in the kernel of $\psi$, essentially by definition. This contradiction shows that $\B_4'$ cannot be generated by two elements, as claimed. This completes the proof of the first statement of the theorem.

\medskip

We now prove the second statement of the theorem. We define
\[
a = \sigma_1\sigma_2^{-1} \qquad b = \sigma_1\sigma_3^{-1}\qquad \text{and}\qquad r = \alpha^2\sigma_1^{-10}.
\]
It suffices to show that the subgroup $\langle a,b,r\rangle$ generated by $a,b$ and $r$ contains all elements of the form $\sigma_1\sigma_i^{-1}$ with $1\leq i \leq 5$. As this subgroup contains the elements $\sigma_1\sigma_2^{-1}$ and $\sigma_1\sigma_3^{-1}$ by definition, it remains only to show that it also contains $\sigma_1\sigma_4^{-1}$ and $\sigma_1\sigma_5^{-1}$. 

We claim that $\sigma_1\sigma_5^{-1}\in \langle a,b,r\rangle$. We calculate that $rbr^{-1} = \sigma_3\sigma_5^{-1}$, and so we have
\[
\sigma_1\sigma_5^{-1} = (\sigma_1\sigma_3^{-1})(\sigma_3\sigma_5^{-1}) = brbr^{-1}.
\]
The claim follows. 

We next claim that $\sigma_1\sigma_4^{-1}\in \langle a,b,r\rangle$. We will prove that 
\[
\sigma_1\sigma_4^{-1} = b^{-9}rar^{-1}b^{10},
\]
and the claim will follow. We calculate that
\begin{align*}
rar^{-1}  = (\alpha^2\sigma_1^{-10})\sigma_1\sigma_2^{-1}(\alpha^2\sigma_1^{-10})^{-1} &= \alpha^2\left(\sigma_1^{-9}\sigma_2^{-1}\sigma_1^{10}\right)\alpha^{-2}\\[.1in]
& = \sigma_3^{-9}\sigma_4^{-1}\sigma_3^{10},
\end{align*}
where in the last equality we have used that $\alpha^2\sigma_i\alpha^{-2} = \sigma_{i+2}$. Since $\sigma_1$ commutes with both $\sigma_3$ and $\sigma_4$, we further have that
\begin{align*}
\sigma_3^{-9}\sigma_4^{-1}\sigma_3^{10} &= \sigma_1^{10}\left(\sigma_3^{-9}\sigma_4^{-1}\sigma_3^{10}\right)\sigma_1^{-10}\\[.1in]
& = (\sigma_1\sigma_3^{-1})^9\left(\sigma_1\sigma_4^{-1}\right)(\sigma_3\sigma_1^{-1})^{10}\\[.1in]
& = b^9(\sigma_1\sigma_4^{-1})b^{-10}.
\end{align*}

\vspace{.1in}

\noindent It now follows that 
\[
rar^{-1}  = b^9(\sigma_1\sigma_4^{-1})b^{-10}
\]

\vspace{.1in}

\noindent and hence that
\[
b^{-9}rar^{-1}b^{10} = \sigma_1\sigma_4^{-1},
\]

\vspace{.1in}

\noindent as desired. This completes the proof of the theorem. 
\end{proof}

\section{Appendix: Another generating set for the commutator subgroup}

In this appendix, we give a self-contained proof of the following.

\begin{proposition}
Let $n\geq 5$. The commutator subgroup $\B_n'$ is generated by the elements $\sigma_1\sigma_i^{-1}$ with $2\leq i\leq n-1$. 
\end{proposition}

As mentioned above, this generating set can also be found be examining the generators and relations in the presentation for $\B_n'$ found by Gorin--Lin~\cite[p.7]{lin}. Our proof will use only the fact that $\B_n'$ is the normal closure in $\B_n$ of either of the elements $\sigma_1\sigma_3^{-1}$ and $\sigma_1\sigma_2^{-1}$ provided $n\geq 5$ (see \cite[Remark 1.10]{lin}).

\begin{proof}

We first prove the proposition for $n \geq 6$. To do this, we will demonstrate the apparently weaker fact that $\B_n'$ is generated by the set of all elements of the form $\sigma_i\sigma_j^{-1}$ with $1\leq i,j\leq n-1$ and $i\neq j$ (here we will regard $i$ and $j$ as elements of the subset $\{1,2,\ldots, n-1\}$ of the integers). The proposition will then follow because any such element can be written as a product of elements of the form $\sigma_1\sigma_i^{-1}$ and their inverses. Indeed, we have for all $i$ and $j$ that 
\[
\sigma_i\sigma_j^{-1} = \sigma_i\sigma_1^{-1}\sigma_1\sigma_j^{-1} = (\sigma_1\sigma_i^{-1})^{-1}(\sigma_1\sigma_j^{-1}).
\]

Since any element of the form $\sigma_i\sigma_j^{-1}$ is conjugate to either $\sigma_1\sigma_2^{-1}$ or $\sigma_1\sigma_3^{-1}$, the group $\B_n'$ is  generated in $\B_n$ by any element of the form $\sigma_a\sigma_b^{-1}$. To prove the proposition, it therefore suffices to prove that any element of the form
\[
g(\sigma_a\sigma_b^{-1})g^{-1}
\]
can be written as a product of elements of the form $\sigma_i\sigma_j^{-1}$. We will prove this by inducting on the length of $g$ with respect to the generators $\sigma_1,\ldots, \sigma_{n-1}$. We say that an element has length $N\geq 0$ if it can be written as a product of $N$ elements of the form $\sigma_i^{\pm 1}$. Any element that is not the identity has positive length. 

\medskip

There is nothing to prove if $g$ has length zero, as the only element of length 0 is the identity. For the base case of the induction, we will prove that any element of the form $\sigma_k^{\epsilon}(\sigma_i\sigma_j^{-1})\sigma_k^{-\epsilon}$ with $\epsilon = \pm 1$ can be expressed in the desired form. 

We begin the proof of the base case with the claim that given any two generators $\sigma_i$ and $\sigma_j$, there exists a generator $\sigma_\ell$ commuting with both. If $j-i\geq 2$ we can take $\ell = i$ since $\sigma_i$ commutes with itself and with $\sigma_j$. Now assume that $j=i+1$. Since $n\geq 6$, there exists $\ell$ such that either $\ell < i$ and $i-\ell \geq 2$ or $\ell > i+ 1$ and $\ell- (i+1) \geq 2$. In either case, $\sigma_{\ell}$ commutes with both $\sigma_i$ and $\sigma_{i+1}$. This proves the claim.

Assume that $\epsilon = 1$. By the preceding claim, there exists $\sigma_{\ell}$ that commutes with both $\sigma_i$ and $\sigma_j$. We then have that
\[
\sigma_k\sigma_i\sigma_j^{-1}\sigma_k^{-1} = (\sigma_k\sigma_{\ell}^{-1})(\sigma_i\sigma_j^{-1})(\sigma_{\ell}\sigma_k^{-1}).
\]
Next, assume that $\epsilon = -1$. Again, by the above claim there exists a generator $\sigma_{\ell}$ that commutes with both $\sigma_k$ and $\sigma_i$ and a generator $\sigma_m$ that commutes with both $\sigma_j$ and $\sigma_k$. We then have that
\[
\sigma_k^{-1}\sigma_i\sigma_j^{-1}\sigma_k = (\sigma_k^{-1}\sigma_i)(\sigma_j^{-1}\sigma_k)= \left((\sigma_\ell\sigma_{k}^{-1})(\sigma_i\sigma_\ell^{-1})\right)\left((\sigma_m\sigma_{j}^{-1})(\sigma_k^{1}\sigma_m^{-1})\right)
\]
This establishes the base case. 

We now assume that there exists $N\geq 1$ such that for any element $g\in \B_n$ of length $\leq N$, and any element of the form $\sigma_a\sigma_b^{-1}$ with $1\leq a,b\leq n-1$ there exists a factorization of the form 
\[
g\left(\sigma_a\sigma_b^{-1}\right)g^{-1} = \prod \sigma_A\sigma_B^{-1}
\]
where the product ranges over a finite collection of elements of the form $\sigma_A\sigma_B^{-1}$ such that $1\leq A,B\leq n-1$. (As in the previous section, products are formed from left to right.)

Let $g$ be a word of length $N+1$. Then there exists a generator $\sigma_i$ and a word $h$ of length $N$ so that either $g = \sigma_i h$ or $g = \sigma_i^{-1} h$. The result now follows by induction, since for any indices $a,b$ we have by induction that 
\begin{align*}
g(\sigma_a\sigma_b^{-1})g^{-1} = (\sigma_i^{\pm 1} h)(\sigma_a\sigma_b^{-1})(\sigma_i^{\pm 1} h)^{-1} &= \sigma_i^{\pm 1}\left(h(\sigma_a\sigma_b^{-1})h^{-1}\right)\sigma_i^{\mp 1}\\
&= \sigma_i^{\pm 1}\left(\prod (\sigma_A\sigma_B^{-1})\right){\sigma_i^{\mp 1}} \\
& = \prod \sigma_i^{\pm 1}(\sigma_A\sigma_B^{-1})\sigma_i^{\mp 1}\\
& = \prod\left( \prod \sigma_{A'}\sigma_{B'}^{-1}\right),
\end{align*}
where we have used in the last equality the fact that each $\sigma_i^{\pm 1}(\sigma_A\sigma_B^{-1})\sigma_i^{\mp 1}$ can be expanded as a product of elements of the form $\sigma_{A'}\sigma_{B'}^{-1}$ with $1\leq A',B' \leq n-1$. This completes the proof of the proposition for $n\geq 6$.

\medskip

For the $n=5$ case of the proposition, a proof essentially identical to the one given for $n\geq 6$ shows that $\B_n'$ is generated by the set of all elements of the form $\sigma_i\sigma_j^{-1}$ where now the indices $i$ and $j$ are allowed to assume the value $0$ (we recall that $\sigma_0 = \alpha\sigma_4\alpha^{-1}$). This difference from the $n\geq 6$ arises because, while none of the generators $\sigma_i$ with $1\leq i\leq 4$ commutes with both $\sigma_2$ and $\sigma_3$, the half-twist $\sigma_0$ does commute with both. 

To complete the proof, all that remains to be shown is that any element of the form $\sigma_i\sigma_0^{-1}$ or $\sigma_0\sigma_i^{-1}$ can be written as a product of elements of the form $\sigma_1\sigma_j^{-1}$ and their inverses. It suffices to do this for the $\sigma_i\sigma_0^{-1}$. Much as in the $n\geq 6$ case, we compute that
\[
\sigma_i\sigma_0^{-1} =\left( \sigma_i\sigma_1^{-1}\right)\left(\sigma_1\sigma_0^{-1}\right).
\]

We claim that $\sigma_1\sigma_0^{-1}$ can be written as a product of elements of the form $\sigma_i\sigma_j^{-1}$ with $i,j\neq 0$. Recalling that $\alpha = \sigma_1\sigma_2\sigma_3\sigma_4$, we calculate that
\[
\sigma_1\sigma_0^{-1} = (\sigma_1^2\sigma_2)(\sigma_3\sigma_4^{-1})(\sigma_3^{-1}\sigma_2^{-1}\sigma_1^{-1}) = (\sigma_1^2\sigma_2\sigma_1^{-3})(\sigma_3\sigma_4^{-1})(\sigma_1^3\sigma_3^{-1}\sigma_2^{-1}\sigma_1^{-1}),
\]
where we have used that the element $\sigma_1^3$ commutes with $\sigma_3\sigma_4^{-1}$. The middle term on the right side of second equality is of the required form, so it remains only to show that both the left-most term and the right-most term can be factored in the desired way. We have that
\[
\sigma_1^2\sigma_2\sigma_1^{-3} = (\sigma_1\sigma_4^{-1})^2(\sigma_2\sigma_4^{-1})(\sigma_4\sigma_1^{-1})^3
\]
and that
\[
\sigma_1^3\sigma_3^{-1}\sigma_2^{-1}\sigma_1^{-1} = (\sigma_1\sigma_3^{-1})(\sigma_1\sigma_4^{-1})^2(\sigma_4\sigma_2^{-1})(\sigma_4\sigma_1^{-1})
\]
as desired. The claim follows, and the proof of the proposition is complete.

%

\end{proof}

\bibliographystyle{plain}
\bibliography{smallgen}

\begin{thebibliography}{1}

\bibitem{baykurkorkmaz}
R.~Inanc {Baykur} and Mustafa {Korkmaz}.
\newblock {The mapping class group is generated by two commutators}.
\newblock {\em arXiv e-prints}, page arXiv:1908.11306, Aug 2019.

\bibitem{chenkordekmargalit}
Lei {Chen}, Kevin {Kordek}, and Dan {Margalit}.
\newblock {Homomorphisms between braid groups}.
\newblock {\em arXiv e-prints}, page arXiv:1910.00712, Oct 2019.

\bibitem{gorinlin}
E.~A. Gorin and V.~Ja. Lin.
\newblock Algebraic equations with continuous coefficients, and certain
  questions of the algebraic theory of braids.
\newblock {\em Mat. Sb. (N.S.)}, 78 (120):579--610, 1969.

\bibitem{lin}
Vladimir {Lin}.
\newblock {Braids and Permutations}.
\newblock {\em arXiv Mathematics e-prints}, page math/0404528, Apr 2004.

\end{thebibliography}

\end{document}